\documentclass[12pt]{amsproc}
\newtheorem{theorem}{\sc Theorem}[section]
\newtheorem{lemma}[theorem]{\sc Lemma}
\newtheorem{assump}{\sc Additional assumption}

\begin{document}

\title[]{Groups with a covering condition on commutators}
\thanks{The first and second authors are members of GNSAGA (INDAM), 
and the third author was  supported by  FAPDF and CNPq.}

\author{Eloisa Detomi}
\address{Dipartimento di Matematica \lq\lq Tullio Levi-Civita\rq\rq, Universit\`a di Padova, Via Trieste 63, 35121 Padova, Italy} 
\email{eloisa.detomi@unipd.it}
\author{Marta Morigi}
\address{Dipartimento di Matematica, Universit\`a di Bologna\\
Piazza di Porta San Donato 5 \\ 40126 Bologna \\ Italy}
\email{marta.morigi@unibo.it}
\author{Pavel Shumyatsky}
\address{Department of Mathematics, University of Brasilia\\
Brasilia-DF \\ 70910-900 Brazil}
\email{pavel@unb.br}
\begin{abstract} Given a group $G$ and positive integers $k,n$, we let $B_n=B_n(G)$ denote the set of all elements $x\in G$ such that $|x^G|\leq n$, and we say that $G$ satisfies the $(k,n)$-covering condition for commutators if there is a subset $S\subseteq G$ such that $|S|\leq k$ and all commutators of $G$ are contained in $S B_n$. The importance of groups satisfying this condition was revealed in the recent study of probabilistically nilpotent finite groups of class two. The main result obtained in this paper is the following theorem.

\noindent Let $G$ be a group satisfying the $(k,n)$-covering condition  for commutators. Then $G'$ contains a characteristic subgroup $B$ such that $[G':B]$ and $|B'|$ are both $(k,n)$-bounded. 

\noindent This extends several earlier results of similar flavour.

\end{abstract}
\maketitle
\section{Introduction}
Let $G$ be a group. If $x\in G$, we write $x^G$ for the conjugacy class containing $x$. Given positive integers $k,n$, we let $B_n=B_n(G)$ denote the set of all elements $x\in G$ such that $|x^G|\leq n$, and we say that $G$ satisfies the $(k,n)$-covering condition  for commutators if there is a subset $S\subseteq G$ such that $|S|\leq k$ and all commutators of $G$ are contained in $S B_n$. For shortness, we will refer to this condition as the  $(k,n)$-covering condition. 

The importance of groups satisfying the $(k,n)$-covering condition was first revealed in \cite{ebes}, where it was shown that a probabilistically nilpotent finite group of class two contains a subgroup of bounded index which satisfies the condition. In fact, a large part of the work \cite{ebes} is devoted to the analysis of such groups. In particular, it was shown that if a (possibly infinite) group $G$ satisfies the $(k,n)$-covering condition, then there is a subgroup $T\leq G$ such that the index $[G:T]$ and the order of $\gamma_4(T)$ both are finite and $(k,n)$-bounded.

Throughout, we write $\gamma_i(K)$ for the $i$-th term of the lower central series of a group $K$. We use the expression ``$(a,b,\dots)$-bounded" to mean that a quantity is bounded from above by a number depending only on the parameters $a,b,\dots$.

In the present paper, we further clarify the structure of groups satisfying the $(k,n)$-covering condition.

\begin{theorem}\label{glav} Let $G$ be a group satisfying the $(k,n)$-covering condition. Then $G'$ contains a characteristic subgroup $B$ such that $[G':B]$ and $|B'|$ are both $(k,n)$-bounded. 
\end{theorem}

This can be seen as an extension of the main result obtained in \cite{aglase}, which says that  if all commutators of a group $G$ are contained in $B_n(G)$ then the second commutator subgroup $G''$ is finite and has $n$-bounded order.

As an immediate corollary of Theorem \ref{glav}   we obtain that, whenever a group $G$ satisfying the $(k,n)$-covering condition, then $G'$ contains a characteristic subgroup $T$ which is nilpotent of class at most $2$ and has  $(k,n)$-bounded index in $G'$. Indeed, assume that Theorem \ref{glav} holds and set $T=C_B(B')$. It is easy to see that $T$ is as required.

A more general covering condition for $w$-values, where $w$ is a multilinear commutator word, was considered in \cite{jadms}. If $w$ is a group-word, say that a group $G$ satisfies the $(k,n)$-covering condition with respect to the word $w$ if the $w$-values of $G$ are contained in $SB_n$ for some subset $S\subseteq G$ such that $|S|\leq k$. In particular, the paper \cite{jadms} contains a proof of the following result.

\begin{theorem}\label{old} \cite{jadms} Let $w$ be a multilinear commutator word on $m$ variables and let $G$ be a group satisfying the $(k,n)$-covering condition with respect to the word $w$. Then $G$ has a soluble subgroup $T$ such that $[G : T]$ and the derived length of $T$ are both $(k,m,n)$-bounded.
\end{theorem}

Observe that our Theorem \ref{glav} cannot be deduced from the above.
% since in Theorem \ref{glav} the derived length of $B$ does not depend on $n$.
 It is absolutely unclear whether under the hypotheses of Theorem \ref{old} the group $G$ contains a subgroup of $(k,m,n)$-bounded index and $(k,m)$-bounded derived length.

\section{Preliminaries}

If $x,y,z$ are elements of a group, the following identities hold. Throughout the paper they will be used without references.

\begin{enumerate}
\item $[x,y]^{-1}=[y,x]$;
\item $[xy,z]=[x,z]^y[y,z]=[x,z][x,z,y][y,z]$;
\item $[x,yz]=[x,z][x,y]^z$.
\end{enumerate}
Given subsets $X,Y$ of a group $G$, we write $Comm(X,Y)$ for the set $\{[x,y] \mid x\in X,y\in Y\}$. Thus, $Comm(X,Y)$ is the set of commutators and $[X,Y]$ denotes the subgroup generated by $Comm(X,Y)$. Note that if $X$ is an abelian normal subgroup and $Y=\{y\}$ contains only one element, then $Comm(X,Y)=[X,y]$.
\begin{lemma}\label{511} Let $t$ be a positive integer, $M$ a normal subgroup of a group $G$, and $a\in G$. Suppose that $G=M\langle a\rangle$. If $|M'|$ and $[M:C_M(a)]$ are both at most $t$, then $|G'|\leq t^2$.
\end{lemma}
\begin{proof} Note that $G'=M'[M,a]$. Pass to the quotient $G/M'$ and assume that $M$ is abelian. Then $[M,a]=Comm(M,a)$ has size at most $t$. The lemma follows. \end{proof}

Let $G$ be a group generated by a set $X$ such that $X=X^{-1}$. For an element $g\in G$, we write $l_X(g)$ to denote the minimal number $l$ with the property that $g$ can be written as a product of $l$ elements of $X$. Clearly, $l_X(g)=0$ if and only if $g=1$. The following lemma is taken from \cite{aglase}.

\begin{lemma}\label{passma} Let $G$ be a group generated by a set $X=X^{-1}$ and let $H$ be a subgroup of finite index $m$ in $G$. Then each coset $Hb$ contains an element $g$ such that $l_X(g)\leq m-1$.
\end{lemma}

\section{Proof of Theorem \ref{glav}}

Throughout this section $G$ is a (possibly infinite) group satisfying the $(k,n)$-covering condition. We wish to show that $G'$ contains a characteristic subgroup $B$ such that $[G':B]$ and $|B'|$ are both $(k,n)$-bounded. 

Consider first the case where $|S|=1$, that is, there is $s\in G$ such that $Comm(G,G)\subseteq s B_n $. As $1 \in s B_n $, it follows that $s\in B_n$ and therefore $Comm(G,G)\subseteq B_{n^2}$. By \cite{aglase} the second commutator subgroup $G''$ has finite $n$-bounded order. Therefore Theorem  \ref{glav} holds with $B=G'$.  Hence, we may assume that $|S|\geq2$ and argue by induction on $|S|$. Also we may assume that $1\in S$. 

Suppose that $x\in Comm(G,G)$ and $n<|x^G|\leq n^{100}$. It follows that $x\in s B_n$ for some nontrivial $s\in S$. We deduce that $|s^G|\leq n^{101}$. Therefore $|g^G|\leq n^{102}$ whenever $g\in s B_n$. We can remove $s$ from the set $S$ while increasing $n$ to $n^{102}$ and use induction.
Thus, without loss of generality, we make the following additional assumptions.

\begin{assump}\label{nice}
If $x\in Comm(G,G)$ and $x\in B_{n^{100}}$, then $x\in B_n$.
\end{assump}

More generally, 

\begin{assump}\label{nice2}
If $x\in Comm(G,G)$ and $x\in sB_{n^{100}}$ for some $s\in S$, then $x\in sB_n$.
\end{assump}
Indeed, we know that there is $s'\in S$ such that $x\in s'B_n$.  So there is $y\in B_n$ such that $s'y\in sB_{n^{100}}$, whence $s'\in sB_{n^{101}}$. It follows that $s'B_n$ is contained in $sB_{n^{102}}$. So increasing $n$ to $n^{102}$ we can remove $s'$ from the set $S$.

Furthermore,

\begin{assump}\label{nice3}
$sB_n\cap s'B_n=\emptyset$ whenever $s,s'\in S$ and $s\neq s'$.
\end{assump}
This is because if $sB_n\cap s'B_n\not=\emptyset$, then $s^{-1}s'\in B_{n^2}$ and so $s'B_n\subset sB_{n^3}$. Hence, increasing $n$ to $n^{3}$ we can remove $s'$ from the set $S$.
\medskip

Set
$X=Comm(G,G)\cap B_n.$ 

If $N$ is a normal subgroup of $G$ and $\bar{G}=G/N$, let $X(\bar{G})$ be the set of commutators $x\in\bar{G}$ such that $|x^{\bar{G}}|\leq n$. We cannot claim that $X(\bar{G})$ is necessarily $\bar{X}$, the image of $X$. On the other hand, it is clear that $\bar{G}$ satisfies the $(k,n)$-covering condition and it is easy to check that if $X(\bar{G})\neq\bar{X}$, then there is $s\in S$ whose image $\bar{s}$ in $\bar{G}$ has at most $n^2$ conjugates. So using a variation of the above arguments we can remove $\bar{s}$ from the set $\bar{S}$ and argue by induction on $|S|$.

Furthermore, the Assumptions \ref{nice}--\ref{nice3} may fail in the quotient $\bar{G}$. However the above argument can be applied to $\bar{G}$ enabling us to remove an element from the set $\bar{S}$ and use induction on $|S|$. 

Hence, we make

\begin{assump}\label{nice4} There is a number $n_0$ depending only on $k$ and $n$ with the property that if $N$ is a normal subgroup of $G$, then either the Assumptions \ref{nice}--\ref{nice3} hold in $\bar{G}$ and $X(\bar{G})=\bar{X}$ or $\bar{G}'$ contains a characteristic subgroup $\bar{B}$ such that $[\bar{G}':\bar{B}]$ and $|\bar{B}'|$ are both at most $n_0$. 
\end{assump}

Given $g\in G$, write 
 \[ C^*_G(g)=\{x\in G \mid [x,g]\in X\}.\]
 
In what follows we denote the subgroup $\langle X\rangle$ by $B$ and show that this subgroup is as claimed in Theorem \ref{glav}.

Note that  in $G/B$ the image of $C^*_G(g)$  is contained in the centralizer of the image of $g$. 

\begin{lemma}\label{subgroup} For any $g\in G$ the set $C^*_G(g)$ is a subgroup of $G$.
\end{lemma}
\begin{proof} Let $g\in G$ and $x,y\in C^*_G(g)$. We need to show that $x^{-1}$ and $xy$ belong to $C^*_G(g)$. Since $[x^{-1},g]=([x,g]^{-1})^{x^{-1}}$ is a conjugate of $[x,g]^{-1}$ and since $[x,g]\in X$, it follows that  $[x^{-1},g]\in X$ and so $x^{-1}\in C^*_G(g)$. Further, write $[xy,g]=[x,g]^y[y,g]$. Since $[x,g]$ and $[y,g]$ both lie in $B_n$, it follows that $[xy,g]\in B_{n^2}$. Because of Assumption \ref{nice} we conclude that $[xy,g]\in B_n$ and so $xy$ belong to $C^*_G(g)$.
\end{proof}

We remark that while the definition of $C^*_G(g)$ makes sense for any group $G$, in general this set is not a subgroup. It is clear that the proof of Lemma \ref{subgroup} uses our specific assumptions about $G$. 

\begin{lemma}\label{coset} For any $g,h\in G$ there is $s\in S$ such that the coset $C^*_G(g)h$ is precisely the set $\{x\in G \mid [x,g]\in sB_n\}$.
\end{lemma}
\begin{proof}  Let $g,h\in G$. In view of Assumption \ref{nice3} there is a unique element $s\in S$ such that $[h,g]\in sB_n$. Choose $y\in C^*_G(g)$. Taking into account that $[y,g]\in B_n$ write 
\[[yh,g]=[y,g]^h[h,g]\in B_nsB_n\subseteq sB_{n^2}.\]
It follows from Assumption \ref{nice2}  that $[yh,g]\in sB_n$. This shows that the coset $C^*_G(g)h$ is contained in the set 
\[\{x\in G \mid  [x,g]\in sB_n\}.\]

To establish the other containment, let $u\in\{x\in G \mid [x,g]\in sB_n\}$. We need to show that $u\in C^*_G(g)h$, that is, $[uh^{-1},g]\in B_n$. Note that $[h^{-1},g]=([h,g]^{-1})^{h^{-1}}$. Write 
\[ [uh^{-1},g]=[u,g]^{h^{-1}}[h^{-1},g]\in s^{h^{-1}}B_nB_n(s^{-1})^{h^{-1}}\subseteq B_{n^2}.\]
 In view of Assumption \ref{nice} the result follows.
\end{proof}

\begin{lemma}\label{index} For any $g\in G$ the index $[G:C^*_G(g)]$ is at most $k$.
\end{lemma}
\begin{proof}  This is immediate from Lemma \ref{coset}.
\end{proof}

Recall that $B=\langle X\rangle$. For any $x\in B$ write $l(x)$ for the minimal number $l$ such that $x$ is a product of $l$ elements from $X$.

\begin{lemma}\label{comms} $Comm(B,G)\subseteq X$.
\end{lemma}
\begin{proof}   Suppose that this is false and there exist $x\in B$ and $y\in G$ with the property that $[x,y]\not\in X$. Choose $x$ such that $l(x)$ is minimal. Suppose that $l(x)=1$. Then $x\in B_n$ and $[x,y]\in B_{n^2}$. In view of Assumption \ref{nice} we conclude that $[x,y]\in X$, a contradiction. Therefore $l(x)\geq2$.
Write $x=x_1x_2$, where $l(x_i)\leq l(x)-1$. We have $[x_i,y]\in X$ and $[x,y]=[x_1,y]^{x_2}[x_2,y]\in B_{n^2}$, a contradiction again. 
\end{proof}

A group is said to be a BFC-group if its conjugacy classes are finite and of bounded size. B. H. Neumann proved that in a BFC-group the derived group $G'$ is finite \cite{bhn}. It follows that if $|x^G|\leq n$ for each $x\in G$, then the order of $G'$ is bounded by a number depending only on $n$. A first explicit bound for the order of $G'$ was found by J. Wiegold \cite{wie}, and the best known was obtained in \cite{gumaroti}. The particular case where $[G:Z(G)]\leq n$ is known as Schur's theorem \cite[10.1.4]{Rob}.

For any subgroup $H\leq G$ write $X_H$ to denote the set of commutators $Comm(H,H)\cap X$. 
 Observe that if  $H$ is normal in $G$, then the subgroup  $\langle X_H \rangle$ is normal as well. 

\begin{lemma}\label{bbbb} 
 Let $H \le G$.  The index $[H' : \langle X_H \rangle]$ is $k$-bounded.
 In particular, the index $[G':B]$ is $k$-bounded.
\end{lemma}
\begin{proof}  Let $B_H= \langle X_H \rangle$. For $h\in H$ note that $C^*_H(h)$ is contained in the centralizer of $h$ modulo $B_H$. Lemma \ref{index} implies that any element of $H/B_H$ has centralizer of index at most $k$. By Neumann's theorem, $H/B_H$ has commutator subgroup of $k$-bounded order. Hence the result.
\end{proof}

\begin{lemma}\label{02} For any $x\in X$ the subgroup $[B,x]^G$ has finite $n$-bounded order.
 \end{lemma}
\begin{proof} 
Choose $x\in X$. 
 First we prove that $[B,x]$ has finite $n$-bounded order. 
 Since $C_B(x)$ has index at most $n$ in $B$, by Lemma \ref{passma} we can choose elements $y_1,\dots,y_n$ such that $l(y_i)\leq n-1$ and $[B,x]$ is generated by the commutators $[y_i,x]$. Observe that by Lemma \ref{comms} the commutators $[y_i,x]$ belong to $X$. For each $i=1,\dots,n$ write $y_i=y_{i1}\dots y_{i(n-1)}$, where $y_{ij}\in X$. The standard commutator identities show that $[y_i,x]$ can be written as a product of conjugates in $B$ of the commutators $[y_{ij},x]$. Let $h_1,\dots,h_t$ be the conjugates in $B$ of elements from the set $\{x,y_{ij} \mid 1\leq i,j\leq n\}$. Since the index of $C_B(h)$ in $B$ is at most $n$ for any $h\in X$, it follows that $t$ here is $n$-bounded. Let $T=\langle h_1,\dots,h_t\rangle$. It is clear that $[B,x]\leq T'$ and so it is sufficient to show that $T'$ has $n$-bounded order. Observe that the index of $C_B(h_i)$ in $B$ is at most $n$ for each $i=1,\dots,t$. Therefore the centre $Z(T)$ has index at most $n^t$ in $T$. Thus, Schur's theorem tells us that the order of $T'$ is $n$-bounded, as required.

 Now, as $x \in X$, we have that $[B,x]^G$ is a product of at most $n$ 
 conjugates of $[B,x]$, normalizing each other and having $n$-bounded order. We conclude that $[B,x]^G$ has finite $n$-bounded order. 
 \end{proof}

Denote by $m$  the maximum of indices of $C_B(x)$ in $B$, where $x\in X$. Obviously, $m\leq n$. Pick $a\in X$ such that $C_B(a)$ has index precisely $m$ in $B$ and choose $b_1,\dots,b_m\in B$ such that $l(b_i)\leq m-1$ and $a^B=\{a^{b_i} \mid i=1,\dots,m\}$. The existence of these elements is guaranteed by Lemma \ref{passma}. 
 Set 
\[U=C_G(\langle b_1,\dots,b_m\rangle).\] 
 Note that the subgroup $U$ has finite $n$-bounded index in $G$. This follows from the facts that $l(b_i)\leq m-1$ and $C_G(x)$ has index at most $n$ in $G$ for every $x\in X$.

\begin{lemma}\label{03}  Suppose that $u\in U$ and $ua\in X$. Then $[B,u]\leq[B,a]$.
\end{lemma}
\begin{proof} Since $u\in U$, it follows that $(ua)^{b_i}=ua^{b_i}$ for each $i=1,\dots,m$. The fact that $ua\in X$ implies that $ua$ has at most $m$ conjugates in $B$. Therefore the elements $ua^{b_i}$ form the conjugacy class $(ua)^B$. For an arbitrary element $g\in B$ there exists $h\in\{b_1,\dots,b_m\}$ such that $(ua)^{g}=ua^{h}$ and so $u^ga^g=ua^h$. Therefore $[u,g]=a^ha^{-g}\in[B,a]$. The lemma follows.
\end{proof} 

 Write $a=[d,e]$ for suitable $d,e\in G$, and let  
 \[ U_1=C_G^*(d)\cap C_G^*(e)\cap U \quad \textrm{ and } \quad U_0=\cap_{g\in G}U_1^g. \] 
  Note that $[U_0,d][U_0,e]\leq B$. Moreover, as $[U:U_1]$ is at most $k^2$, the index of $U_0$ in $G$ is $(k,n)$-bounded.  
 Set  
\[ N=[B,a]^G, \] 
and note that $N$ has finite  $n$-bounded order by Lemma \ref{02}. 

\begin{lemma}\label{04}  The subgroups  $[B,[U_0,d]]$, $[B,[U_0,e]]$, and $[B,X_{U_0}]$ are contained in $N$.
\end{lemma} 
\begin{proof} Choose $h\in U_0$. To establish the containment  $[B,[U_0,d]]\leq N$ we need to show that $[B,[h,d]]\leq N$.

Write 
\[ [d,eh]^{h^{-1}}=[d,h]^{h^{-1}}[d,e].\]
 Both $[d,h]$ and $[d,e]$ lie in $X$ so $[d,eh]\in B_{n^2}$. Assumption \ref{nice} now guarantees that $[d,eh]\in X$. Denote $[d,h]^{h^{-1}}$ by $u$ and deduce from Lemma \ref{03} that $[B,[d,h]]\leq N$. Since $[B,[d,h]]=[B,[h,d]]$, we conclude that $[B,[U_0,d]]\leq N$. By a symmetric argument $[B,[U_0,e]]\leq N$.

To prove that $[B,X_{U_0}]$ is contained in $N$ choose elements $h_1,h_2\in U_0$ such that $[h_1,h_2]\in X_{U_0}$.  Write 
\[ [h_1d,eh_2]=[h_1,h_2]^d[d,h_2][h_1,e]^{dh_2}[d,e]^{h_2}\]
 and so 
 \[ [h_1d,eh_2]^{h_2^{-1}}=[h_1,h_2]^{dh_2^{-1}}[d,h_2]^{h_2^{-1}}[h_1,e]^d[d,e].\]
  The commutators on the right hand side belong to $X$ so, because of Assumption \ref{nice}, we conclude that $[h_1d,eh_2]\in X$. Denote the product $[h_1,h_2]^{dh_2^{-1}}[d,h_2]^{h_2^{-1}}[h_1,e]^d$ by $v$. Thus, the right hand side of the above equality is $va$. Taking into account that $U_0$ is a normal subgroup, check that $v\in U_0$. By Lemma \ref{03}, $[B,v]\leq[B,a]$. We know that the subgroups  $[B,[U_0,d]]$ and $[B,[U_0,e]]$ are contained in $N$ so it follows that $[B,[h_1,h_2]]\leq N$ and we deduce that $[B,X_{U_0}]\leq N$.
\end{proof} 

Let $i$ be a positive integer. We say that a normal subgroup $L\leq G$ is a special $(i)$-subgroup (or simply $(i)$-subgroup) if the following conditions hold.
\begin{enumerate}
\item $B\leq L$;
\item $X_L\subseteq Z(B)$;
\item $[G:L]\leq i$.
\end{enumerate}

Note that the group $G$ is its own special $(i)$-subgroup for some $i\geq1$ if and only if $B$ is abelian. In this case $G$ is a special $(1)$-subgroup.
\begin{lemma}\label{zabyl}  There is a $(k,n)$-bounded number $j_0$ with the property that $G$ has a normal subgroup $T$ of finite $(k,n)$-bounded order such that $G/T$ possesses a $(j_0)$-subgroup.
\end{lemma}
\begin{proof} Set $L=U_0B$. Since the index of $U_0$ in $G$ is $(k,n)$-bounded, we can find $(k,n)$-boundedly many commutators $c_1,\dots,c_t\in X$ such that $L=U_0\langle c_1,\dots,c_t\rangle$.
 Let $T$ be  the product of the subgroups $N=[B,a]^G$ and $[B,c_i]^G$ for $i=1,\dots,t$. By Lemma \ref{02} each of these subgroups has $n$-bounded order. 
 At $t$ is $(k,n)$-bounded, we  conclude that $T$ has finite $(k,n)$-bounded order. 

Pass to the quotient $G/T$.  Denote the images of $G$, $B$, and $L$ in $G/T$ by same symbols. Taking into account Assumption \ref{nice4}, we simply assume that $T=1$. In that case the commutators $c_1,\dots,c_t$ are contained in $Z(B)$ and therefore $X_L\subseteq X_{U_0}Z(B)$. In view of Lemma \ref{04}, more specifically the statement on $[B,X_{U_0}]$, we have $X_{U_0}\subseteq Z(B)$. Therefore $X_L\subseteq Z(B)$. Obviously the index, say $j_0$, of $L$ in $G$ is $(k,n)$-bounded. The proof is complete.
\end{proof}

\begin{lemma}\label{blin} Suppose that $G$ has a $(j)$-subgroup $L$ for some $j\geq2$. Then $G$ has a normal subgroup $T$ of $(j,k,n)$-bounded order such that $G/T$ possesses a $(j-1)$-subgroup.
\end{lemma}
\begin{proof} First, we note that if $B$ is abelian, then the group $G$ is its own $(j-1)$-subgroup (as well as $(1)$-subgroup) and so the result holds. We therefore assume that $B$ is not abelian.

Recall that $a=[d,e]\in X$ and $C_B(a)$ has maximal possible index in $B$. If both $d$ and $e$ belong to $L$, we conclude (since $X_L\subseteq Z(B)$) that $B$ is abelian, contrary to our assumptions. Thus, assume that at least one of them, say $d$, is not in $L$.

Set $V=U_0\cap L$. By Lemma \ref{04}  $[B,[V,d]]\leq N$, where $N=[B,a]^G$. Let $C=C^*_L(d)=\{l\in L \mid  [l,d]\in X\}$. Observe that $V\leq C$ and let $g_1,\dots,g_t$ be a full system of representatives of the right cosets of $V$ in $C$. Note that $t$ here is $(k,n)$-bounded. The subgroup $[C,d]$ is generated by $[V,d]^{g_1},\dots,[V,d]^{g_t}$ and $[g_1,d],\dots,[g_t,d]$. This is straightforward from the fact that $[vg,d]=[v,d]^g[g,d]$ for any $g,v\in G$.
 Let 
 \[ J=L\langle d\rangle \quad \textrm{ and }  \quad B_0=\langle X_L, Comm(C,d)\rangle.\] 
 Here  $B_0$ is normal in $J$ because so is $\langle X_L\rangle$ and by Lemma \ref{comms} $B$ is central in $L$ modulo $\langle X_L\rangle$. 
  Moreover, note that the order of $L'/B_0$ and the index of $C_{L/B_0}(d)$ in $L/B_0$ are both $(k,n)$-bounded. The fact that the former is bounded is immediate from Lemma \ref{bbbb}. The latter is bounded because so is the index of $C$ in $L$. Therefore by Lemma \ref{511} the index of $B_0$ in $J'$ is $k$-bounded. Hence, we can choose $(k,n)$-boundedly many commutators $c_1,\dots,c_{t'}\in X_J$ such that $\langle X_J\rangle$ is generated by $B_0$ and $c_1,\dots,c_{t'}$.

 Let $R$ be the product of  $N$ and the subgroups $[B,[g_i,d]]^G$ for $i=1,\dots,t$ as well as $[B,c_i]^G$ for $i=1,\dots,t'$. 
 Observe that $X_J$ is central in $B$ modulo $R$. 
 In view of Lemma \ref{02} the order of the normal subgroup $R$ is $(k,n)$-bounded. 
 Moreover, since $J$ contains $L$ properly, the index of $J$ in $G$ is at most $j-1$. Taking into account Assumption \ref{nice4} we conclude that if $J$ is normal in $G$, then the image of $J$ in $G/R$ is a $(j-1)$-subgroup. So in the case where $J$ is normal the lemma is proved.

Now suppose that $J$ is not normal. Choose $g\in G$ such that $[d,g]\not\in J$. Since, by Lemma \ref{bbbb}, the index $[G':B]$ is $k$-bounded, we can choose $k$-boundedly many conjugates of $[d,g]$, say $a_1,\dots,a_{m'}$, such that $L\langle a_1,\dots,a_{m'}\rangle$ is normal. Set $A=\langle a_1,\dots,a_{m'}\rangle$ and $B_1=A\cap B$. Since the index $[A:B_1]$ is $k$-bounded, we conclude that $B_1$ can be generated by $k$-boundedly many elements, say $a'_1,\dots,a'_r$. Then $[B_1,A]$ is generated by $(k,n)$-boundedly many conjugates of the commutators $[a'_i,a_l]$ for $1\leq i\leq r$ and $1\leq l\leq{m'}$, as the commutators $[a'_i,a_l]$ are contained in $X$ by Lemma \ref{comms}. The image of $B_1$ in $A/[B_1,A]$ is central and so by Schur's theorem the commutator subgroup of $A/[B_1,A]$ has $k$-bounded order. We conclude that the subgroup $\langle X_A\rangle$ is generated by $(k,n)$-boundedly many commutators from $X_A$. By Lemma \ref{02} for each of such generators there is a normal subgroup of bounded order, modulo which the element is central in $B$. Therefore $G$ has a normal subgroup $N_1$ of $(k,n)$-bounded order such that $\langle X_A\rangle$ is central in $B$ modulo $N_1$.

Recall that by Lemma \ref{04} $[B,[U_0,d]]\leq N$. Since the subgroups $B,V$, and $N$ are normal in $G$ and since the elements $a_i$ are conjugates of $[d,g]$, it follows that $[B,[V,a_i]]\leq N$ for every $i=1,\dots,{m'}$. More generally, $[B,[V,y]]\leq N$ for every $y\in A$. Let $h_1,\dots,h_{k'}$  be a full system of representatives of the right cosets of $B_1$ in $A$. Here of course $k'$ is $k$-bounded. 
 We claim that for each $i\le k'$ there exists a normal subgroup $R_i$, of $(k,n)$-bounded order, such that the set $Comm(L,h_i)\cap X$ is central in $B$ modulo $R_i$. For shortness, set $h=h_i$ and for each right coset $Vy_j$ of $V$ in $L$ such that $Comm(Vy_j,h)\cap X$ is nonempty choose a representative $y_j$ such that $[y_j,h]\in X$. As $V=U_0\cap L$ has $(k,n)$-bounded index in $L$, there are only $(k,n)$-boundedly many such representatives $y_1,\dots,y_s$.  Define the normal subgroup $R_i$ as the product of $N$ and the subgroups $[B,[y_j,h]]^G$ for $j=1,\dots,s$ and note that in view of Lemma \ref{02} the order of $R_i$ is $(k,n)$-bounded. Let $z\in L$ such that $[z,h]\in X$, then $z=vy_j$ for some $v\in V$ and some $y_j$. Then $[B,[z,h]]=[B,[vy_j,h]]\le [B,[v,h]]^G[B,[y_j,h]]^G\le R_i$.
Therefore $Comm(L,h_i)\cap X$ is central in $B$ modulo $R_i$ and the claim is proved. 

Choose $x\in L$ and $y\in A$ such that 
\[[x,y]\in Comm(L,A)\cap X.\]
Write $y=b h_i$, where $i\leq k'$ and $b\in B_1$. We have $[x,y]=[x,b h_i]=[x, h_i][x, b]^{h_i}$.  
 By Lemma \ref{comms}  $[x,b]\in X$. We deduce from Assumption \ref{nice} that $[x,h_i]\in X$.
 Furthermore, $[x,b]\in Z(B)$ because $X_L\subseteq Z(B)$ while the commutator $[x,h_i]$ is central in $B$ modulo $R_i$. We conclude that the set $Comm(L,A)\cap X$ is central in $B$ modulo $\prod_{i\leq k'} R_i$. Obviously, the order of $\prod_{i\leq k'} R_i$ is $(k,n)$-bounded.

Let  
 \[ E=\langle X_L,X_A,Comm(L,A)\cap X\rangle \quad \textrm{ and } \quad K=LA.\] 
 We claim that $E$ is normal in $K$.
 Indeed, the subgroup generated by $X_L$ is normal n $G$. Moreover, $X_L,X_A$ and $Comm(L,A)\cap X$ are contained in $B$ and, by virtue of Lemma \ref{comms},  $B$ is central in $L$ modulo $\langle X_L\rangle$. Furthermore, the set $Comm(L,A)\cap X$ is closed under conjugation by elements of $A$. Now, let us show that the order of $K'/E$ is $(k,n)$-bounded. Note that $K'=L'[L,A]A'$. Lemma \ref{bbbb} shows that $L'$ and $A'$ both have $(k,n)$-bounded orders modulo $E$ so we only need to check that the order of the image of $[L,A]$ in $K/E$ is $(k,n)$-bounded. Recall that $A=\langle a_1,\dots,a_{m'}\rangle$ and observe that $[L,A]=\prod_{1\leq i\leq m'}[L,a_i]$. 
 As $C^*_L(a_i)$  has  index at most $k$ in  $L$ and its image in $K/E$ is contained in $C_{L/E}(a_iE)$, by Lemma \ref{511} the image of each factor $[L,a_i]$ in $K/E$ has $(k,n)$-bounded order.  
 We conclude that indeed the order of $K'/E$ is $(k,n)$-bounded.

Therefore there are $(k,n)$-boundedly many commutators $x_1,\dots,x_u$ in $X_K$ such that $\langle X_K\rangle$ is generated by $x_1,\dots,x_u$ and
 $E$. 
  We let $T_0$ be  the product of the normal subgroups $N$, $N_1$, all subgroups $R_i$ for $i\leq k'$, and $[B,x_i]^G$ for $i=1,\dots,u$.
   In view of Lemma \ref{02} we observe that the order of $T_0$ is $(k,n)$-bounded. Further,  $X_K$ is central in $B$ modulo $T_0$. Note that, since $K$ contains $L$ properly, the index of $K$ in $G$ is at most $j-1$. Thus, the proof is complete.
\end{proof}

We now have at our disposal all the necessary tools to complete a proof of Theorem \ref{glav}.
\begin{proof} [Proof of Theorem \ref{glav}] In view of Lemma \ref{bbbb} it is sufficient to show that $B'$ has finite $(k,n)$-bounded order. Lemma \ref{zabyl} tells us that there is a $(k,n)$-bounded number $j_0$ with the property that $G$ has a normal subgroup $T_0$ of $(k,n)$-bounded order such that $G/T_0$ possesses a $(j_0)$-subgroup $L_0/T_0$.
 If $L_0=G$, then the image of $B$ in $G/T_0$ is abelian and so $B'\leq T_0$. In this case $B'$ has finite $(k,n)$-bounded order and we have nothing to prove. Therefore assume that $L_0<G$ and $j_0\geq2$. Keeping in mind Assumption \ref{nice4} pass to the quotient $G/T_0$. By Lemma \ref{blin} the group $G/T_0$ has a normal subgroup $T_1/T_0$ of $(j_0,k,n)$-bounded order such that $G/T_1$ possesses a $(j_1)$-subgroup $L_1/T_1$, where $j_1\leq j_0-1$. 
 
 Now we repeat the same argument with $j_1, L_1, T_1$ in place of $j_0, L_0, T_0$ respectively and we continue recursively.   

 Therefore we  find normal subgroups $T_0<T_1<T_2<\dots$, all of bounded order, and special subgroups $L_0<L_1<L_2<\dots$ 
 until, after at most $j_0-1$ steps, we reach the conclusion that $G$ has a normal subgroup $T$ of finite $(k,n)$-bounded order such that $B'\leq T$. This completes the proof.
\end{proof}

\end{document}